\documentclass{amsart}

\usepackage[dvips]{graphicx}
\usepackage{latexsym,amsfonts} 
\usepackage{amsmath,amssymb,graphics,setspace}
\usepackage{amsthm}
\usepackage{mathrsfs} 

\usepackage{graphics}
\usepackage{paralist}

\usepackage{marvosym}

\usepackage{pstricks,pst-text,pst-grad,pst-node,pstricks-add,pst-poly,pst-coil,pst-arrow} 
\usepackage{pstricks,pst-3dplot,graphicx}

\usepackage{hyperref}
\hypersetup{linktocpage=true,colorlinks=true,linkcolor=blue,citecolor=blue,pdfstartview={XYZ 1000 1000 1}}

\newcommand{\assign}{\mathrel{\mathop :}=}
\newcommand{\vor}{\mbox{vor}} 
\newcommand{\Nrv}{\mbox{Nrv}} 
\newcommand{\abs}[1]{\left|#1\right|}     
\newcommand{\cl}{\mbox{c$\ell$}}  
\newcommand{\Int}{\mbox{int}} 
\newcommand{\bdy}{\mbox{bdy}} 
\newcommand{\cyc}{\mbox{cyc}} 
\newcommand{\hcyc}{\mbox{hCyc}}
\newcommand{\hvor}{\mbox{hVor}}
\newcommand{\bhcyc}{\mbox{bhCyc}}
\newcommand{\hNrv}{\mbox{hNrv}}
\newcommand{\vorNrv}{\mbox{vorNrv}}
\newcommand{\hvorNrv}{\mbox{hVorNrv}}
\newcommand{\br}{\mbox{br}}
\newcommand{\hbr}{\mbox{hBr}}
\newcommand{\sh}{\mbox{sh}} 

\theoremstyle{plain}
\newtheorem{theorem}{Theorem}
\newtheorem{lemma}{Lemma}
\newtheorem{remark}{Remark}
\newtheorem{definition}{Definition}

\newtheorem{example}{Example}

\usepackage{appendix}

\begin{document}



\title[Homotopic Nerve Complexes]{Homotopic Nerve Complexes with\\ Free Group Presentations}

\author[J.F. Peters]{J.F. Peters}
\address{
Computational Intelligence Laboratory,
University of Manitoba, WPG, MB, R3T 5V6, Canada and
Department of Mathematics, Faculty of Arts and Sciences, Ad\.{i}yaman University, 02040 Ad\.{i}yaman, Turkey,
}
\email{james.peters3@umanitoba.ca}

\thanks{The research has been supported by the Natural Sciences \&
Engineering Research Council of Canada (NSERC) discovery grant 185986 
and Instituto Nazionale di Alta Matematica (INdAM) Francesco Severi, Gruppo Nazionale per le Strutture Algebriche, Geometriche e Loro Applicazioni grant 9 920160 000362, n.prot U 2016/000036 and Scientific and Technological Research Council of Turkey (T\"{U}B\.{I}TAK) Scientific Human
Resources Development (BIDEB) under grant no: 2221-1059B211301223.}


\date{}

\dedicatory{Dedicated to Yuriy Trokhymchuk \& Saroja V. Banavar}

\begin{abstract}
This paper introduces homotopic nerve complexes in a planar Whitehead CW space and their Rotman free group presentations.
Nerve complexes were introduced by P.S. Alexandrov during the 1930s and recently given a formal structure from a computational topology perspective by H. Edelsbrunner and J.L. Harer in 2010.  A homotopic nerve results from the nonvoid intersection of a collection of homotopic 1-cycles.  Briefly, a 1-cycle is a finite sequence of path-connected vertexes with no end vertex and with a nonvoid interior. A homotopic 1-cycle has the structure of a 1-cycle in a CW space in which cycle edges are replaced by homotopic maps.  A group $G(V,+)$ containing a basis $\mathcal{B}$ is {\em free}, provided every member of $V$ can be written as a linear combination of elements (generators) of the basis $\mathcal{B}\subset V$.  Let $\bigtriangleup$ be the members $v$ of $V$, each written as a linear combination of the basis elements of $\mathcal{B}$.  A presentation of $G(V,+)$ is a mapping $\mathcal{B}\times \bigtriangleup\to G(\left\{v\in V\assign\sum_{k\in \mathbb{Z}\atop
g\in \mathcal{B}}{kg}\right\},+)$.  The main results in this paper are (1) Every homotopic vortex nerve has a free group presentation and (2) For a vortex nerve that consists of a finite collection of closed, convex sets in Euclidean space, the nerve and union of sets in the nerve have the same homotopy type. 
\end{abstract}
\maketitle
\tableofcontents

\maketitle

\begin{figure}[!ht]
\centering
\begin{pspicture}
(-1.0,-1.0)(3.5,4.1)
\centering
\psframe[linewidth=0.75pt,linearc=0.25,cornersize=absolute,linecolor=blue](-2.0,-1.0)(11.0,4.2)
\psline[linewidth=0.5pt,linecolor=blue]%
(-1.0,2)(-0.5,3)
\psline[linewidth=0.5pt,linecolor=blue]%
(-0.5,3)(0.5,3.5)
\psline[linewidth=0.5pt,linecolor=blue]%
(0.5,3.5)(1.5,3.0)
\psline*[linecolor=gray!20]
(1.5,3.0)(2.0,2.0)(0.5,1.5)
\psline[linewidth=0.5pt,linecolor=blue]%
(1.5,3.0)(2.0,2.0)
\psdots[dotstyle=o,dotsize=2.5pt,linewidth=1.2pt,linecolor=black,fillcolor=red!80](0.5,1.5)
\psline[linewidth=0.5pt,linecolor=blue]%
(2.0,1.0)(1.5,0.0)
\psline[linewidth=0.5pt,linecolor=blue]%
(1.5,0.0)(0.5,-0.5)
\psline[linewidth=0.5pt,linecolor=blue]%
(-0.5,0.0)(-1.0,1.0)
\psline[linewidth=0.5pt,linecolor=blue]%
(-1.0,1.0)(-1.0,2)
\psdots[dotstyle=o,dotsize=2.5pt,linewidth=1.2pt,linecolor=black,fillcolor=red!80]
(-1.0,2)(-0.5,3)(0.5,3.5)(1.5,3.0)(2.0,2.0)(2.0,1.0)(1.5,0.0)%
(0.5,-0.5)(-0.5,0.0)(-1.0,1.0)(-1.0,1.0)(-1.0,2)
\rput(-0.7,2.0){\footnotesize $\boldsymbol{v_1}$}
\rput(-0.7,1.0){\footnotesize $\boldsymbol{v_0}$}
\psdots[dotstyle=o,dotsize=1.5pt,linewidth=1.2pt,linecolor=black,fillcolor=red!80]
(2.0,1.7)(2.0,1.5)(2.0,1.3)
\psdots[dotstyle=o,dotsize=1.5pt,linewidth=1.2pt,linecolor=black,fillcolor=red!80]
(0.2,-0.4)(0.0,-0.3)(-0.2,-0.2)
\rput(-1.5,3.85){\footnotesize $\boldsymbol{K}$}
\rput(0.5,3.7){\footnotesize $
\boldsymbol{v_3}$}
\rput(1.8,1.0){\footnotesize $
\boldsymbol{v_i}$}
\rput(-0.1,0.2){\footnotesize $
\boldsymbol{v_{(n-1)[n]}}$}
\rput(1.2,0.2){\footnotesize $
\boldsymbol{v_{i+1}}$}
\rput(0.5,2.5){$
\boldsymbol{\cyc E}$}
\rput(2.5,1.5){\large $
\boldsymbol{\to}$}
\end{pspicture}\hfil
\begin{pspicture}
(-3.0,-1.0)(3.0,4.0)
\centering
\psline[linewidth=0.5pt,linecolor=blue,arrowscale=1.0]{-v}%
(-1.0,2)(-0.5,3)
\psline[linewidth=0.5pt,linecolor=blue,arrowscale=1.0]{-v}%
(-0.5,3)(0.5,3.5)
\psline[linewidth=0.5pt,linecolor=blue,arrowscale=1.0]{-v}%
(0.5,3.5)(1.5,3.0)
\psline*[linecolor=gray!20]
(1.5,3.0)(2.0,2.0)(0.5,1.5)
\psline[linewidth=0.5pt,linecolor=blue,arrowscale=1.0]{-v}%
(1.5,3.0)(2.0,2.0)
\psdots[dotstyle=o,dotsize=2.5pt,linewidth=1.2pt,linecolor=black,fillcolor=red!80](0.5,1.5)
\psline[linewidth=0.5pt,linecolor=blue,arrowscale=1.0]{-v}%
(2.0,1.0)(1.5,0.0)
\psline[linewidth=0.5pt,linecolor=blue,arrowscale=1.0]{-v}%
(1.5,0.0)(0.5,-0.5)
\psline[linewidth=0.5pt,linecolor=blue,arrowscale=1.0]{-v}%
(-0.5,0.0)(-1.0,1.0)
\psline[linewidth=0.5pt,linecolor=blue,arrowscale=1.0]{-v}%
(-1.0,1.0)(-1.0,2)
\psdots[dotstyle=o,dotsize=2.5pt,linewidth=1.2pt,linecolor=black,fillcolor=red!80]
(-1.0,2)(-0.5,3)(0.5,3.5)(1.5,3.0)(2.0,2.0)(2.0,1.0)(1.5,0.0)%
(0.5,-0.5)(-0.5,0.0)(-1.0,1.0)(-1.0,1.0)(-1.0,2)
\rput(-0.7,2.0){\footnotesize $\boldsymbol{v_1}$}
\rput(-0.7,1.0){\footnotesize $\boldsymbol{v_0}$}
\psdots[dotstyle=o,dotsize=1.5pt,linewidth=1.2pt,linecolor=black,fillcolor=red!80]
(2.0,1.7)(2.0,1.5)(2.0,1.3)
\psdots[dotstyle=o,dotsize=1.5pt,linewidth=1.2pt,linecolor=black,fillcolor=red!80]
(0.2,-0.4)(0.0,-0.3)(-0.2,-0.2)
%
\rput(2.8,0.9){\footnotesize %
$\boldsymbol{h_{i}(0)=v_i}$} 
\rput(2.5,0){\footnotesize $\boldsymbol{h_{i}(1)=v_{i+1}}$}
\rput(-1.0,-0.2){\footnotesize $
\boldsymbol{h_{(n-1)[n]}(0)}$}
\rput(-2.3,0.9){\footnotesize %
$\boldsymbol{h_{0}(0) =h_{n-1[n]}(1)}$} 
\rput(-1.85,0.4){\footnotesize $\boldsymbol{=\ v_0}$}%
\rput(-2.2,2.1){\footnotesize $
\boldsymbol{h_1(0)\ = h_0(1)}$}
\rput(-1.6,3.1){\footnotesize $
\boldsymbol{h_2(0)=h_{1}(1)}$}
\rput(0.9,3.7){\footnotesize $
\boldsymbol{h_3(0) = v_3}$}
\rput(0.5,2.5){$
\boldsymbol{\hcyc E}$}
\end{pspicture}
\caption[]{1-Cycle $\boldsymbol{\cyc E\ \to}$ homotopic cycle $\boldsymbol{\hcyc E}$ on space $\boldsymbol{K}$.}
\label{fig:homotopicCycle}
\end{figure}

\section{Introduction}
This paper is based on~\cite{Peters2021KievConf}, which introduces homotopic nerve complexes in a planar Whitehead CW space~\cite[\S 4-5]{Whitehead1949BAMS-CWtopology} and their Rotman free group presentations~\cite[\S 11,p.239]{Rotman1965theoryOfGroups}.  A CW complex in a space $K$ is a closure-finite cell complex that is Hausdorff (distinct points in $K$ are in disjoint neighborhoods), satisfying the containment property (closure of every cell complex is in $K$) and intersection property (common parts of cell complexes in $K$ are also in $K$).  A complex $K$ is locally finite, i.e., every point $p\in K$ is a member of some finite subcomplex of $K$ and every complex has a finite number of faces~\cite[\S 5.2,p.65]{Switzer2002CWcomplex}.  A planar CW complex $K$ is a collection of 0-cells (vertexes), 1-cells (edges) and 2-cells (filled triangles).  Collections of planar cells attached to each other are sub-complexes in $K$. For the details, see Appendix~\ref{ap:CW}.

\section{Preliminaries}  
The foundations for homotopic nerves are briefly introduced in this section, starting with the definition of a 1-cycle.

\begin{definition}\label{def:filledCycle} {\rm \bf 1-Cycle}.
In a CW space $K$~\cite{Whitehead1949BAMS-CWtopology}, a 1-cycle $E$ (denoted by $\cyc E$) in a CW space $K$ is a collection of path-connected vertexes (0-cells) on edges (1-cells) attached to each other with no end vertex and $\cyc E$ has a nonvoid interior.
\qquad\textcolor{blue}{\Squaresteel}
\end{definition}

\begin{definition}\label{def:CycleClosure} {\rm \bf Closure}.
Let $A\in 2^X$ (nonvoid subset $A$ in a space $X$) and $D(x,A) = \inf\left\{a\in A: \abs{x-a}\right\}$ be the Hausdorff distance between a point $x\in A$ and subset $A$~\cite[\S 22,p. 128]{Hausdorff1914b}.  The closure of $A$~\cite[\S 1.18,p. 40]{DBLP:series/isrl/2014-63} is defined by
\[
\cl(A) = \left\{x\in X: D(x,A) = 0\right\}.\ \mbox{\qquad\textcolor{blue}{\Squaresteel}}
\]
\end{definition}

\begin{definition}\label{def:closure} For a space $X, A\in 2^X$, let $\cl A$ be the closure of $A$.  Then the boundary of $A$ (denoted by $\bdy A$) is the set of all points on the border of $\cl A$ and not in the complement of $\cl A$ (denoted by $\partial \cl A$).  Also, the interior of $A$ (denoted by $\Int A$) is the set of all points in $\cl A$ and not on the boundary of $A$\footnote{This version of $\Int A$ comes from T. Vergili.}, {\em i.e},
\begin{align*}
\partial(\cl A) &= X\setminus \cl A,\ \mbox{all points in $X$ and not in $\cl A$}.\\
\Int(A) &= \left\{E\in 2^X: E\subset \cl A\ \mbox{and}\ E \cap \bdy A = \emptyset  \right\}.\\
\bdy(A) &= X\setminus (\Int A\cup \partial \cl A).\\ 
\cl A &= \bdy(A)\cup\Int(A).\ \mbox{\qquad\textcolor{blue}{\Squaresteel}}
\end{align*}
\end{definition} 

\begin{example}\label{ex:1-cycle}
A sample 1-cycle $\cyc E$ with vertices $v_0,\dots,v_{(n-1)[n]}$ with $[n] = \mbox{mod}\ n, n\in\mathbb{Z}^{0,+}$, is shown in Fig.~\ref{fig:homotopicCycle}.  This 1-cycle $\cyc E$ has a finite number of $i\ \mbox{mod}\ n$ vertices, with the initial vertex $v_0 = v_{n\ \mbox{mod}\ n}$.  $\cyc E$ also has a nonvoid interior (denoted by $\Int(\cyc E)$) containing at least one filled triangle with one of its vertices equal to centroid of the cycle and its other two vertices on the boundary of $\cyc E$ (denoted by $\bdy(\cyc E)$).  In other words, 
\[
\cl(\cyc E) = \bdy(\cyc E)\cup \Int(\cyc E).\mbox{\qquad\textcolor{blue}{\Squaresteel}}
\]
\end{example}

The edges in a 1-cycle cell complex in a CW space are replaced by homotopic maps to obtain a homotopic cycle.

\begin{definition}\label{def:HomotopicMap} {\rm \bf Homotopic-Map}.\\
Let $I = [0,1]$, the unit interval.  A {\bf homotopic map}(aka\ \emph{\bf path}) in a space $X$ is a continuous map $h:I\to X$ with endpoints $h(0)=x_0, h(t)\in [h(0),h(1)], t\in I$ and $h(1)=x_1$~\cite[\S 2.1,p.11]{Switzer2002CWcomplex}.  
\qquad \textcolor{blue}{\Squaresteel}
\end{definition}


\begin{definition}\label{def:filledHomotopicCycle} {\rm \bf Homotopic-Cycle~\cite[App. A]{PetersVergili2021proximalHomotopy}}.\\
A {\bf\large homotopic cycle} $\hcyc E$ in a CW space $K$ is a sequence of continuous maps $\left\{h_i\right\}_{i=0}^{(n-1)[n]}$, $h_i:I\to K, i\in\mathbb{Z}, I = [0,1], [n] =\ \mbox{mod}\ n$, with 
\begin{align*}
h_0(0) &= h_{n[n]}(0) = v_0\in K, h_{n-1[n]}(1)= v_{n-1[n]}\in K.\\
\mbox{Given}\ g &\in \left\{h_i\right\}_{i=0}^{(n-1)[n]},\\ 
g &= h_0(0),\ \mbox{a basis element in $\hcyc E$,  so that}\\
v &\overbrace{\assign kg =kh_0(0)=h_{0+k}(0)=h_{k}(0)\in K, k\in \mathbb{Z}.}^{\mbox{\textcolor{blue}{\bf $h_{i+k}(0)$\ \mbox{walks forward}\ $k$ vertexes\ \mbox{from $g$ to reach}\ $v$}}}\\
\bar{h}_{i+k}(0)&\overbrace{\assign h_{(i+k)-k}(0) = h_i(0).}^{\mbox{\textcolor{blue}{\bf $\bar{h}_{i+k}(0)$\ \mbox{walks back}\ $k$ vertexes from $h_{i+k}(0)$ to reach $h_i(0)$}}}\\
+:K\times K&\to K,\ \mbox{is defined by}\\
+(v,v') &= \overbrace{+(kg,k'g) = kg+k'g.}^{\mbox{\textcolor{blue}{\bf move to $\boldsymbol{v''}$ via $\boldsymbol{v,v'}$}}}\\
        &= kh_0(0)+k'h_0(0) = h_{(k+k')[n]}(0) = v''\in K.\mbox{\qquad \textcolor{blue}{\Squaresteel}}
\end{align*}
\end{definition}

In a homotopic cycle $\hcyc E$, every vertex $v_i$ is reachable by $k$ maps from a distinguished vertex $h_0(0)=v_0$ (a generator in the free group presentation of the homotopic cycle).  An inverse homotopic map $\bar{h}_{i+k}(0)$ is defined in terms of reverse movement (traversal from the $h_{i+k}(0)$ cycle vertex to the previous $h_{i}(0)$ cycle vertex)\footnote{Here, inverses are defined in terms of homotopic map indices, instead of $\bar{h}_i(t) = h_i(1-t)$~\cite[\S 1.1, p. 27]{Hatcher2002CUPalgebraicTopology}. Many thanks to Tane Vergili for pointing this out.}.  For simplicity, we consider only $i\in \mathbb{Z}$ for movements between cycle vertexes.  

\begin{example}
A planar homotopic cycle over a collection of vertexes in a CW space is shown in Fig.~\ref{fig:homotopicCycle}.  From Def.~\ref{def:filledHomotopicCycle}, the expression $3v_0 = 3h_0(0) = 
 h_{0+3}(0) = v_3\in K$ represents a move from vertex $v_0$ to vertex $v_3$.  What we observed about the structure of the 1-cycle $\cyc E$ in Example~\ref{ex:1-cycle} carries over in delineating the structure of a homotopic cycle.  That is, $\hcyc E$ also has a nonvoid interior (denoted by $\Int(\hcyc E)$) that it inherits from 1-cycle $\cyc E$. The boundary of $\hcyc E$ (denoted by $\bdy(\hcyc E)$) has been constructed with its sequence of continuous homotopic maps.  In other words, 
\[
\cl(\hcyc E) = \bdy(\hcyc E)\cup \Int(\hcyc E).\mbox{\qquad\textcolor{blue}{\Squaresteel}}
\]
\end{example}

\begin{theorem}\label{thm:JordanCurveTheorem} {\rm [Jordan Curve Theorem~\cite{Jordan1893coursAnalyse}]}.\\
	A simple closed curve lying on the plane divides the  
	plane into two regions and forms their common boundary.
\end{theorem}
\begin{proof}
	For the first complete proof, see O. Veblen~\cite{Veblen1905TAMStheoryOFPlaneCurves}.  For a simplified proof via the Brouwer Fixed Point Theorem, see R. Maehara~\cite{Maehara1984AMMJordanCurvedTheoremViaBrouwerFixedPointTheorem}.  For an elaborate proof, see J.R. Munkres~\cite[\S 63, 390-391, Theorem 63.4]{Munkres2000}.
\end{proof}

The boundary region of any planar homotopic cycle $\hcyc E$ (denoted by $\partial(\hcyc E)$) contains all cell complexes that are not part of the cycle in a CW space. 

\begin{theorem}
Every homotopic cycle in a CW space 
\begin{compactenum}[1$^o$]
\item is a finite collection of path-connected vertices and 
\item the boundary of a homotopic cycle partitions the CW space into two regions.
\end{compactenum}
\end{theorem}
\begin{proof}$\mbox{}$\\
$1^o$: Let $\hcyc E$ be a homotopic cycle in a CW space $K$. From Def.~\ref{def:filledHomotopicCycle}, the end-points $h_{i[n]}(0),h_{i[n]}(1)$ are points in $K$, which are 0-cells (vertices) in $K$.  The 0-cells are path-connected, since each $h_{i+1[n]}(0) = h_{i[n]}(1)$ and $h_{0[n]}(0) = h_{n-1[n]}(1)$. The number of vertices in $\hcyc E$ is restricted by the choice of the modulus $n$.  Hence, $\hcyc E$ has a finite collection of path-connected vertices.\\
$2^o$: From Def.~\ref{def:HomotopicMap}, every homotopic map $h:I\to K$ is continuous.  Also, from Def.~\ref{def:HomotopicMap}, the sequence of homotopic maps in a homotopic cycle $\hcyc E$ on $K$ constructs a planar, simple, closed curve.  Then, from Theorem~\ref{thm:JordanCurveTheorem} (Jordan Curve Theorem), $\hcyc E$ separates the space $K$ into two regions, namely,
\begin{compactenum}[1$^o$]
\item[\bf Region outside $\hcyc E$]: $\partial(\hcyc E) = K\setminus \hcyc E$, and
\item[\bf Region inside $\hcyc E$]:
\[ 
\Int(\hcyc E) = K\setminus (\bdy(\hcyc E)\cup \partial(\hcyc E)).
\]
\end{compactenum} 
\end{proof}

\begin{figure}[!ht]
\centering
\begin{pspicture}
(-3.5,-1.0)(8.0,4.0)
\centering
\psline*[linecolor=gray!20]
(1.5,3.0)(2.0,2.0)(0.5,1.5)
\psline*[linecolor=gray!20]
(2,2)(2.5,3)(4,2.5)
\psline[linewidth=0.5pt,linecolor=blue,arrowscale=1.0]
(-1.0,2)(-0.5,3)
\psline[linewidth=0.5pt,linecolor=blue,arrowscale=1.0]
(-0.5,3)(0.5,3.5)
\psline[linewidth=0.5pt,linecolor=blue,arrowscale=1.0]
(0.5,3.5)(1.5,3.0)
\psline[linewidth=0.5pt,linecolor=blue,arrowscale=1.0]%
(1.5,3.0)(2.0,2.0)
\psline[linewidth=0.5pt,linecolor=blue,arrowscale=1.0]%
(2.0,1.0)(1.5,0.0)
\psline[linewidth=0.5pt,linecolor=blue,arrowscale=1.0]%
(1.5,0.0)(0.5,-0.5)
\psline[linewidth=0.5pt,linecolor=blue,arrowscale=1.0]%
(-0.5,0.0)(-1.0,1.0)
\psline[linewidth=0.5pt,linecolor=blue,arrowscale=1.0]%
(-1.0,1.0)(-1.0,2)
\psdots[dotstyle=o,dotsize=2.5pt,linewidth=1.2pt,linecolor=black,fillcolor=red!80]
(-1.0,2)(-0.5,3)(0.5,3.5)(1.5,3.0)(2.0,2.0)(2.0,1.0)(1.5,0.0)%
(0.5,-0.5)(-0.5,0.0)(-1.0,1.0)(-1.0,1.0)(-1.0,2)
\rput(-0.7,2.0){\footnotesize $\boldsymbol{v_1}$}
\rput(-0.7,1.0){\footnotesize $\boldsymbol{v_0}$}
\psdots[dotstyle=o,dotsize=1.5pt,linewidth=1.2pt,linecolor=black,fillcolor=red!80]
(2.0,1.7)(2.0,1.5)(2.0,1.3)
\psdots[dotstyle=o,dotsize=1.5pt,linewidth=1.2pt,linecolor=black,fillcolor=red!80]
(0.2,-0.4)(0.0,-0.3)(-0.2,-0.2)
\psline[linewidth=0.5pt,linecolor=blue,arrowscale=1.0]
(2,2)(2.5,3)
\psline[linewidth=0.5pt,linecolor=blue,arrowscale=1.0]
(2.5,3)(3.5,3.5)
\psline[linewidth=0.5pt,linecolor=blue,arrowscale=1.0]
(3.5,3.5)(4.5,3.5)
\psline[linewidth=0.5pt,linecolor=blue,arrowscale=1.0]
(4.5,3.5)(5.5,3.0)
\psline[linewidth=0.5pt,linecolor=blue,arrowscale=1.0]
(5.5,3.0)(6.0,2.0)
\psline[linewidth=0.5pt,linecolor=blue,arrowscale=1.0]
(6.0,2.0)(5.0,1.5)
\psline[linewidth=0.5pt,linecolor=blue,arrowscale=1.0]
(5.0,1.5)(4.0,1.5)
\psline[linewidth=0.5pt,linecolor=blue,arrowscale=1.0]
(4.0,1.5)(3.0,1.5)
\psline[linewidth=0.5pt,linecolor=blue,arrowscale=1.0]
(3.0,1.5)(2.0,2)
\psdots[dotstyle=o,dotsize=2.5pt,linewidth=1.2pt,linecolor=black,fillcolor=red!80]
(2.5,3)(3.5,3.5)(5.5,3.0)(6.0,2.0)(5.0,1.5)(4.0,1.5)%
(3.0,1.5)(2.0,2)
\psdots[dotstyle=o,dotsize=3.5pt,linewidth=1.2pt,linecolor=black,fillcolor=green!80]
(2.0,2)
\rput(2.8,2.15){\footnotesize %
$\boldsymbol{v'_0 = v_5}$} 
\rput(6.3,2.0){\footnotesize %
$\boldsymbol{v'_{j}}$}
\rput(0.5,1.5){$
\boldsymbol{\cyc E}$} 
\rput(4.0,2.5){$
\boldsymbol{\cyc E'}$} 
\rput(2.3,0.9){\footnotesize %
$\boldsymbol{v_{i}}$} 
\rput(-0.1,3.0){\footnotesize $
\boldsymbol{v_2}$}
\rput(0.9,3.5){\footnotesize $
\boldsymbol{v_3}$}
\rput(1.7,3.0){\footnotesize $
\boldsymbol{v_4}$}
\rput(2.75,2.8){\footnotesize $
\boldsymbol{v'_1}$}
\end{pspicture}
\caption[]{Nerve $\boldsymbol{\Nrv E = \left\{\cyc E,\cyc E': \cyc E\cap\cyc E'\neq \emptyset\right\}}$.}
\label{fig:nerveComplex}
\end{figure}

\begin{definition}\label{def:filledNervComplex} {\rm \bf Nerve Complex}.\\
A nerve complex $\Nrv E$ in a CW space is a collection of nonempty cell complexes with nonvoid intersection.  This is an example of an Alexandrov nerve~\cite[\S 4.3,p. 39]{Alexandroff1932elementaryConcepts}.
\textcolor{blue}{\Squaresteel}
\end{definition}

A vertex with a collection of triangles attached to it, is called the {\bf nucleus} of the attached triangles (also called the nucleus of an Alexandrov-Hopf nerve complex or {\bf Nerv}~\cite{AlexandroffHopf1935Topologie}).

\begin{lemma}\label{lemma:AlexandrovNerv}
Every vertex in the triangulation of the vertices in a CW space is the nucleus of an Alexandrov nerve complex. 
\end{lemma}
\begin{proof}
Let $K$ be a CW space containing $n$ trianulated vertexes.
Each vertex $v\in K$ has 1 or more triangles attached to it, forming a cell complex $E$. Then $E$ is a collection of cell complexes (filled triangles) with nonvoid intersection.  By Def.~\ref{def:filledNervComplex}, $E$ is an Alexandrov nerve complex.  Hence, each $E\in 2^K$ is the nucleus of an Alexandrov nerve.
\end{proof}

\begin{theorem}\label{thm:Nerv}
A CW space containing $n$ triangulated vertexes contains $n$ Alexandrov-Hopf nerve complexes. 
\end{theorem}
\begin{proof}
Immediate from Lemma~\ref{lemma:AlexandrovNerv}.
\end{proof}

Alexandrov nerves are also found in collections of intersecting 1-cycles.

\begin{theorem}\label{theorem:commonVertex}
A pair of pair of 1-cycles with a common vertex in a CW space is a nerve complex.
\end{theorem}
\begin{proof}
Let $\cyc E,\cyc E'\in 2^k$ be 1-cycles in a CW space K. By definition, $\cyc E,\cyc E'$ are cell complexes in $K$.  Also, let $v\in K$ be a vertex and let $\cyc E\cap \cyc E'=\left\{v\right\}$.  Hence, by Def.~\ref{def:filledNervComplex}, $\Nrv = \cyc E\cup\cyc E'$ is a nerve complex, which is also an Alexandroff nerve or \emph{Nerv}~\cite{AlexandroffHopf1935Topologie}.
\end{proof}

\begin{example}
A pair of intersecting 1-cycles $\cyc E,\cyc E'$ in a nerve complex is shown in Fig.~\ref{fig:nerveComplex}.  
In this nerve complex, 
\begin{align*}
\cyc E\cap\cyc E' &= \left\{v_5\right\}\in \cyc E\\
                  &= \left\{v'_0\right\}\in \cyc E'.\mbox{\qquad \textcolor{blue}{\Squaresteel}}
\end{align*}
\end{example}

\begin{figure}[!ht]
\centering
\begin{pspicture}
(-3.5,-1.0)(8.0,4.0)
\centering
\psline*[linecolor=gray!20]
(1.5,3.0)(2.0,2.0)(0.5,1.5)
\psline*[linecolor=gray!20]
(2,2)(2.5,3)(4,2.5)
\psline[linewidth=0.5pt,linecolor=blue,arrowscale=1.0]{-v}
(-1.0,2)(-0.5,3)
\psline[linewidth=0.5pt,linecolor=blue,arrowscale=1.0]{-v}
(-0.5,3)(0.5,3.5)
\psline[linewidth=0.5pt,linecolor=blue,arrowscale=1.0]{-v}
(0.5,3.5)(1.5,3.0)
\psline[linewidth=0.5pt,linecolor=blue,arrowscale=1.0]{-v}%
(1.5,3.0)(2.0,2.0)
\psline[linewidth=0.5pt,linecolor=blue,arrowscale=1.0]{-v}%
(2.0,1.0)(1.5,0.0)
\psline[linewidth=0.5pt,linecolor=blue,arrowscale=1.0]{-v}%
(1.5,0.0)(0.5,-0.5)
\psline[linewidth=0.5pt,linecolor=blue,arrowscale=1.0]{-v}%
(-0.5,0.0)(-1.0,1.0)
\psline[linewidth=0.5pt,linecolor=blue,arrowscale=1.0]{-v}%
(-1.0,1.0)(-1.0,2)
\psdots[dotstyle=o,dotsize=2.5pt,linewidth=1.2pt,linecolor=black,fillcolor=red!80]
(-1.0,2)(-0.5,3)(0.5,3.5)(1.5,3.0)(2.0,2.0)(2.0,1.0)(1.5,0.0)%
(0.5,-0.5)(-0.5,0.0)(-1.0,1.0)(-1.0,1.0)(-1.0,2)
\rput(-0.7,2.0){\footnotesize $\boldsymbol{v_1}$}
\rput(-0.7,1.0){\footnotesize $\boldsymbol{v_0}$}
\psdots[dotstyle=o,dotsize=1.5pt,linewidth=1.2pt,linecolor=black,fillcolor=red!80]
(2.0,1.7)(2.0,1.5)(2.0,1.3)
\psdots[dotstyle=o,dotsize=1.5pt,linewidth=1.2pt,linecolor=black,fillcolor=red!80]
(0.2,-0.4)(0.0,-0.3)(-0.2,-0.2)
\psline[linewidth=0.5pt,linecolor=blue,arrowscale=1.0]{-v}
(2,2)(2.5,3)
\psline[linewidth=0.5pt,linecolor=blue,arrowscale=1.0]{-v}
(2.5,3)(3.5,3.5)
\psline[linewidth=0.5pt,linecolor=blue,arrowscale=1.0]{-v}
(3.5,3.5)(4.5,3.5)
\psline[linewidth=0.5pt,linecolor=blue,arrowscale=1.0]{-v}
(4.5,3.5)(5.5,3.0)
\psline[linewidth=0.5pt,linecolor=blue,arrowscale=1.0]{-v}
(5.5,3.0)(6.0,2.0)
\psline[linewidth=0.5pt,linecolor=blue,arrowscale=1.0]{-v}
(6.0,2.0)(5.0,1.5)
\psline[linewidth=0.5pt,linecolor=blue,arrowscale=1.0]{-v}
(5.0,1.5)(4.0,1.5)
\psline[linewidth=0.5pt,linecolor=blue,arrowscale=1.0]{-v}
(4.0,1.5)(3.0,1.5)
\psline[linewidth=0.5pt,linecolor=blue,arrowscale=1.0]{-v}
(3.0,1.5)(2.0,2)
\psdots[dotstyle=o,dotsize=2.5pt,linewidth=1.2pt,linecolor=black,fillcolor=red!80]
(2.5,3)(3.5,3.5)(5.5,3.0)(6.0,2.0)(5.0,1.5)(4.0,1.5)%
(3.0,1.5)(2.0,2)
\psdots[dotstyle=o,dotsize=3.5pt,linewidth=1.2pt,linecolor=black,fillcolor=green!80]
(2.0,2)
\rput(3.1,2.0){\footnotesize %
$\boldsymbol{h'_{0}(0)=h_5(0)}$} 
\rput(6.4,2.0){\footnotesize %
$\boldsymbol{h'_{j}(0)}$}
\rput(0.5,1.5){$
\boldsymbol{\hcyc E}$} 
\rput(4.0,3.0){$
\boldsymbol{\hcyc E'}$} 
\rput(3.1,0.9){\footnotesize %
$\boldsymbol{h_{i}(0) = h_{i-1}(1)}$} 
\rput(2.1,0){\footnotesize $\boldsymbol{h_{i}(1)}$}
\rput(-2.5,0.9){\footnotesize %
$\boldsymbol{h_{0}(0) =h_{(n-1)[n]}(1)}$} 
\rput(-2.55,0.6){\footnotesize $\boldsymbol{=\ v_0}$}%
\rput(-2.2,2.1){\footnotesize $
\boldsymbol{h_1(0)\ = h_0(1)}$}
\rput(-1.6,3.1){\footnotesize $
\boldsymbol{h_2(0)=h_{1}(1)}$}
\rput(0.9,3.7){\footnotesize $
\boldsymbol{h_3(0)}$}
\rput(1.95,3.0){\footnotesize $
\boldsymbol{h_4(0)}$}
\end{pspicture}
\caption[]{homotopic nerve $\boldsymbol{\hNrv E = \left\{\hcyc E,\hcyc E': \hcyc E\cap\hcyc E'\neq \emptyset\right\}}$.}
\label{fig:homotopicNerve}
\end{figure}

\begin{definition}\label{def:homotopicNerve} {\rm \bf Homotopic Nerve}.
A homotopic nerve $E$ (denoted by $\hNrv E$) is a collection of homotoptic cycles with nonvoid intersection.
\textcolor{blue}{\Squaresteel}
\end{definition}

\begin{theorem}\label{thm:homotopicNerve}
Every collection of homotopic cycles with a common vertex in a CW space is a homotopic nerve complex.
\end{theorem}
\begin{proof}
Symmetric with the proof of Theorem~\ref{theorem:commonVertex}.
\end{proof}

\begin{example}
In Fig.~\ref{fig:homotopicNerve}, homotopic cycles $\hcyc E,\hcyc E'$ have vertex $h_5(0)$ in common.  Hence, by Theorem~\ref{thm:homotopicNerve}, $\hcyc E\cup\hcyc E'$ is a homotopic nerve complex. \textcolor{blue}{\Squaresteel}
\end{example}

\begin{lemma}\label{lemma:homotopicNerv}
Every vertex in the triangulation of the vertices in a CW space is the nucleus of an homotopic nerve. 
\end{lemma}
\begin{proof}
Symmetric with the proof of Lemma~\ref{lemma:AlexandrovNerv}.
\end{proof}

\begin{theorem}\label{thm:homotopicNerv}
A CW space containing $n$ triangulated vertexes contains $n$ homotopic nerves. 
\end{theorem}
\begin{proof}
Immediate from Lemma~\ref{lemma:homotopicNerv}.
\end{proof}

\begin{figure}[!ht]
\centering
\begin{pspicture}
(-3.5,-1.0)(4.0,4.0)
\centering
\psframe[linewidth=0.75pt,linearc=0.25,cornersize=absolute,linecolor=blue](-3.8,-1.0)(4.8,4.0)
\psline[linewidth=0.5pt,linecolor=blue,arrowscale=1.0]{-v}%
(-1.0,2)(-0.5,3)
\psline[linewidth=0.5pt,linecolor=blue,arrowscale=1.0]{-v}%
(-0.5,3)(0.5,3.5)
\psline[linewidth=0.5pt,linecolor=blue,arrowscale=1.0]{-v}%
(0.5,3.5)(1.5,3.0)
\psline*[linecolor=gray!20]
(1.5,3.0)(2.0,2.0)(0.5,1.5)
\psline[linewidth=0.5pt,linecolor=blue,arrowscale=1.0]{-v}%
(1.5,3.0)(2.0,2.0)
\psline[linewidth=0.5pt,linecolor=blue,arrowscale=1.0]{-v}%
(1.5,3.0)(0.5,1.5)
\psline[linewidth=0.5pt,linecolor=blue,arrowscale=1.0]{-v}%
(0.5,1.5)(2.0,2.0)
\psline[linewidth=0.5pt,linecolor=blue,arrowscale=1.0]{-v}%
(0.5,1.5)(2.0,1.0)
\psline[linewidth=0.5pt,linecolor=blue,arrowscale=1.0]{-v}%
(1.5,0.0)(0.5,1.5)
\psline[linewidth=0.5pt,linecolor=blue,arrowscale=1.0]{-v}%
(0.5,1.5)(0.5,-0.5)
\psline[linewidth=0.5pt,linecolor=blue,arrowscale=1.0]{-v}%
(1.5,0.0)(0.5,1.5)
\psline[linewidth=0.5pt,linecolor=blue,arrowscale=1.0]{-v}%
(0.5,1.5)(-1.0,1.0)
\psline[linewidth=0.5pt,linecolor=blue,arrowscale=1.0]{-v}%
(-1.0,2.0)(0.5,1.5)
\psline[linewidth=0.5pt,linecolor=blue,arrowscale=1.0]{-v}%
(0.5,1.5)(-0.5,3.0)
\psline[linewidth=0.5pt,linecolor=blue,arrowscale=1.0]{-v}%
(0.5,1.5)(0.5,3.5)
\psline[linewidth=0.5pt,linecolor=blue,arrowscale=1.0]{-v}%
(-0.5,0.0)(0.5,1.5)
\psdots[dotstyle=o,dotsize=2.8pt,linewidth=1.2pt,linecolor=black,fillcolor=green!80](0.5,1.5)
\psline[linewidth=0.5pt,linecolor=blue,arrowscale=1.0]{-v}%
(2.0,1.0)(1.5,0.0)
\psline[linewidth=0.5pt,linecolor=blue,arrowscale=1.0]{-v}%
(1.5,0.0)(0.5,-0.5)
\psline[linewidth=0.5pt,linecolor=blue,arrowscale=1.0]{-v}%
(-0.5,0.0)(-1.0,1.0)
\psline[linewidth=0.5pt,linecolor=blue,arrowscale=1.0]{-v}%
(-1.0,1.0)(-1.0,2)
\psdots[dotstyle=o,dotsize=2.5pt,linewidth=1.2pt,linecolor=black,fillcolor=red!80]
(-1.0,2)(-0.5,3)(0.5,3.5)(1.5,3.0)(2.0,2.0)(2.0,1.0)(1.5,0.0)%
(0.5,-0.5)(-0.5,0.0)(-1.0,1.0)(-1.0,1.0)(-1.0,2)
\rput(-0.7,2.0){\footnotesize $\boldsymbol{v_1}$}
\rput(-0.7,0.9){\footnotesize $\boldsymbol{v_0}$}
\psdots[dotstyle=o,dotsize=1.5pt,linewidth=1.2pt,linecolor=black,fillcolor=red!80]
(2.0,1.7)(2.0,1.5)(2.0,1.3)
\psdots[dotstyle=o,dotsize=1.5pt,linewidth=1.2pt,linecolor=black,fillcolor=red!80]
(0.2,-0.4)(0.0,-0.3)(-0.2,-0.2)
\rput(-3.2,3.8){\footnotesize $\boldsymbol{K}$}
\rput(3.3,0.9){\footnotesize %
$\boldsymbol{h_{i}(0) = h_{i-1}(1)}$} 
\rput(2.1,0){\footnotesize $\boldsymbol{h_{i}(1)}$}
\rput(-2.5,0.9){\footnotesize %
$\boldsymbol{h_{0}(0) =h_{(n-1)[n]}(1)}$} 
\rput(-2.55,0.6){\footnotesize $\boldsymbol{=\ v_0}$}%
\rput(-2.2,2.1){\footnotesize $
\boldsymbol{h_1(0)\ = h_0(1)}$}
\rput(-1.6,3.1){\footnotesize $
\boldsymbol{h_2(0)=h_{1}(1)}$}
\rput(0.9,3.7){\footnotesize $
\boldsymbol{h_3(0) = v_3}$}
\rput(-1.0,-0.2){\footnotesize $
\boldsymbol{h_{(n-1)[n]}(0)}$}
\rput(3.0,3.2){\colorbox{gray!20}{$
\boldsymbol{\hcyc E}$}}
\rput(1.95,3.0){\footnotesize $
\boldsymbol{h_4(0)}$}
\rput(2.4,2.0){\footnotesize $
\boldsymbol{h_5(0)}$}
\rput(0.8,1.8){\footnotesize $
\boldsymbol{c}$}
\end{pspicture}
\caption[]{Paths on a triangulated $\boldsymbol{\hcyc E}$.}
\label{fig:homotopicCycle3}
\end{figure}

\begin{definition}\label{def:hTriangle} {\bf Homotopic triangle}. Given a set $V$ of 3 vertexes in a CW space, a sequence of homotopic maps, one for each pair of vertices in $V$, constructs a homotopic triangle $E$ (denoted by $\bigtriangleup_h E$).\qquad\textcolor{blue}{\Squaresteel}
\end{definition}

An obvious place to look for homotopic triangles is in terms of the centroid of a homotopic cycle.

\begin{example}
Let 
\begin{align*}
h_c(0) &= h_4(0)\ \mbox{vertex on $\hcyc E$},\\ 
h_c(1) &= c\ \mbox{centroid of $\bigtriangleup_h$},\\
h_{c'}(0) &= c,\\
h_{c'}(1) &= h_5(0).
\end{align*}
for homotopic maps $h_c,h_{c'}$ from vertex $h_4(0) = h_c(0)$ to centroid $c$ of a homotopic cycle and from $h_{c'}(0) = c$ to $h_{c'}(1) = h_5(0)$ as shown in Fig.~\ref{fig:homotopicCycle3}.  This results in the homotopic
triangle 
\[
\bigtriangleup_h h_4(0)ch_5(0) = \bigtriangleup_h h_c(0)ch_{c'}(1).\mbox{\qquad\textcolor{blue}{\Squaresteel}}
\]
\end{example}

\begin{theorem}\label{thm:hCycTriangles}
Every triangulated homotopic cycle constructs an Alexandrov-Hopf nerve.
\end{theorem}
\begin{proof}
Let $\hcyc E$ be a homotopic cycle derived from the set of vertices $V$ along the boundary of a 1-cycle $\cyc E$. Also let $c$ be centroid of $\cyc E$ and introduce a homotopic map from $h:V\times I\to c$, defined by $h_v(0) = v\in V$ and $h_v(1) = c$.  From Def.~\ref{def:hTriangle}, the collection of homotopic maps in $\hcyc E$ combined with the homotopic $h:I\to c$ construct a collection of homotopic triangles $T$, which have the centroid $c\in\Int(\hcyc E)$ in common.  Hence, by Def.~\ref{def:filledNervComplex}, $T$ is an Alexandroff nerve.
\end{proof}

\begin{theorem}
Every collection of homotopic triangles with nonempty intersection
is an Alexandrov nerve.
\end{theorem}
\begin{proof}
Given a collection $E$ of homotopic triangles with common vertex $h_i(0)$, then by Def.~\ref{def:filledNervComplex}, $T$ is an Alexandroff nerve.
\end{proof}

\begin{example}
A sample collection of homotopic triangles in an Alexandroff nerve is shown in Fig.~\ref{fig:homotopicCycle3}.\qquad\textcolor{blue}{\Squaresteel}
\end{example}

\begin{figure}[!ht]
\centering
\begin{pspicture}
(-3.5,-1.0)(4.0,4.0)
\centering
\psframe[linewidth=0.75pt,linearc=0.25,cornersize=absolute,linecolor=blue](-3.8,-1.0)(4.8,4.0)
\psline[linewidth=0.5pt,linecolor=blue,arrowscale=1.0]{-v}%
(-1.0,2)(-0.5,3)
\psline[linewidth=0.5pt,linecolor=blue,arrowscale=1.0]{-v}%
(-0.5,3)(0.5,3.5)
\psline[linewidth=0.5pt,linecolor=blue,arrowscale=1.0]{-v}%
(0.5,3.5)(1.5,3.0)
\psline*[linecolor=gray!20]
(1.5,3.0)(2.0,2.0)(0.5,1.5)
\psline[linewidth=0.5pt,linecolor=magenta,arrowscale=1.0]{-v}%
(-0.3,0.85)(-0.5,1.5)
\psline[linewidth=0.5pt,linecolor=magenta,arrowscale=1.0]{-v}%
(-0.5,1.5)(-0.4,2.25)
\psline[linewidth=0.5pt,linecolor=magenta,arrowscale=1.0]{-v}%
(-0.4,2.25)(0.1,2.95)
\psline[linewidth=0.5pt,linecolor=magenta,arrowscale=1.0]{-v}%
(0.1,2.95)(0.85,2.95)
\psline[linewidth=0.5pt,linecolor=magenta,arrowscale=1.0]{-v}%
(0.85,2.95)(1.5,2.25)
\psline[linewidth=0.5pt,linecolor=magenta,arrowscale=1.0]{-v}%
(1.5,0.75)(0.9,0.25)
\psline[linewidth=0.5pt,linecolor=magenta,arrowscale=1.0]{-v}%
(1.5,2.25)(1.5,0.85)
\psline[linewidth=0.5pt,linecolor=magenta,arrowscale=1.0]{-v}%
(0.9,0.25)(-0.3,0.85)
\psdots[dotstyle=o,dotsize=2.5pt,linewidth=1.2pt,linecolor=black,fillcolor=green!80]
(-0.3,0.85)(-0.5,1.5)(-0.4,2.25)(0.1,2.95)(0.85,2.95)%
(1.5,2.25)(1.5,0.75)(0.9,0.25)
\psline[linewidth=0.5pt,linecolor=blue,arrowscale=1.0]{-v}%
(1.5,3.0)(2.0,2.0)
\psline[linewidth=0.5pt,linecolor=blue,arrowscale=1.0]{-v}%
(1.5,3.0)(0.5,1.5)
\psline[linewidth=0.5pt,linecolor=blue,arrowscale=1.0]{-v}%
(0.5,1.5)(2.0,2.0)
\psline[linewidth=0.5pt,linecolor=blue,arrowscale=1.0]{-v}%
(0.5,1.5)(2.0,1.0)
\psline[linewidth=0.5pt,linecolor=blue,arrowscale=1.0]{-v}%
(1.5,0.0)(0.5,1.5)
\psline[linewidth=0.5pt,linecolor=blue,arrowscale=1.0]{-v}%
(0.5,1.5)(0.5,-0.5)
\psline[linewidth=0.5pt,linecolor=blue,arrowscale=1.0]{-v}%
(1.5,0.0)(0.5,1.5)
\psline[linewidth=0.5pt,linecolor=blue,arrowscale=1.0]{-v}%
(0.5,1.5)(-1.0,1.0)
\psline[linewidth=0.5pt,linecolor=blue,arrowscale=1.0]{-v}%
(-1.0,2.0)(0.5,1.5)
\psline[linewidth=0.5pt,linecolor=blue,arrowscale=1.0]{-v}%
(0.5,1.5)(-0.5,3.0)
\psline[linewidth=0.5pt,linecolor=blue,arrowscale=1.0]{-v}%
(0.5,1.5)(0.5,3.5)
\psline[linewidth=0.5pt,linecolor=blue,arrowscale=1.0]{-v}%
(-0.5,0.0)(0.5,1.5)
\psdots[dotstyle=o,dotsize=2.8pt,linewidth=1.2pt,linecolor=black,fillcolor=green!80](0.5,1.5)
\psline[linewidth=0.5pt,linecolor=blue,arrowscale=1.0]{-v}%
(2.0,1.0)(1.5,0.0)
\psline[linewidth=0.5pt,linecolor=blue,arrowscale=1.0]{-v}%
(1.5,0.0)(0.5,-0.5)
\psline[linewidth=0.5pt,linecolor=blue,arrowscale=1.0]{-v}%
(-0.5,0.0)(-1.0,1.0)
\psline[linewidth=0.5pt,linecolor=blue,arrowscale=1.0]{-v}%
(-1.0,1.0)(-1.0,2)
\psdots[dotstyle=o,dotsize=2.5pt,linewidth=1.2pt,linecolor=black,fillcolor=red!80]
(-1.0,2)(-0.5,3)(0.5,3.5)(1.5,3.0)(2.0,2.0)(2.0,1.0)(1.5,0.0)%
(0.5,-0.5)(-0.5,0.0)(-1.0,1.0)(-1.0,1.0)(-1.0,2)
\rput(-0.7,2.0){\footnotesize $\boldsymbol{v_1}$}
\rput(-0.7,0.9){\footnotesize $\boldsymbol{v_0}$}
\psdots[dotstyle=o,dotsize=1.5pt,linewidth=1.2pt,linecolor=black,fillcolor=red!80]
(2.0,1.7)(2.0,1.5)(2.0,1.3)
\psdots[dotstyle=o,dotsize=1.5pt,linewidth=1.2pt,linecolor=black,fillcolor=red!80]
(0.2,-0.4)(0.0,-0.3)(-0.2,-0.2)
\rput(-3.2,3.8){\footnotesize $\boldsymbol{K}$}
\rput(3.3,0.9){\footnotesize %
$\boldsymbol{h_{i}(0) = h_{i-1}(1)}$} 
\rput(2.1,0){\footnotesize $\boldsymbol{h_{i}(1)}$}
\rput(-2.5,0.9){\footnotesize %
$\boldsymbol{h_{0}(0) =h_{(n-1)[n]}(1)}$} 
\rput(-2.55,0.6){\footnotesize $\boldsymbol{=\ v_0}$}%
\rput(-2.2,2.1){\footnotesize $
\boldsymbol{h_1(0)\ = h_0(1)}$}
\rput(0.9,3.7){\footnotesize $
\boldsymbol{h_3(0) = v_3}$}
\rput(-1.0,-0.2){\footnotesize $
\boldsymbol{h_{(n-1)[n]}(0)}$}
\rput(-2.5,3.2){\colorbox{gray!20}{$
\boldsymbol{\bhcyc E'}$}}
\psline[linewidth=1.2pt,linestyle=dotted,linecolor=blue,arrowscale=0.5]{-v}%
(-2.5,3.0)(-0.35,2.5)
\rput(-3.0,1.5){\colorbox{gray!20}{$
\boldsymbol{\hcyc E}$}}
\psline[linewidth=1.2pt,linestyle=dotted,linecolor=blue,arrowscale=0.5]{-v}%
(-2.6,1.5)(-1.0,1.5)
\rput(1.95,3.0){\footnotesize $
\boldsymbol{h_4(0)}$}
\rput(2.4,2.0){\footnotesize $
\boldsymbol{h_5(0)}$}
\rput(0.8,1.8){\footnotesize $
\boldsymbol{c}$}
\end{pspicture}
\caption[]{Barycentric homotopic cycle $\boldsymbol{\bhcyc E}$.}
\label{fig:homotopicCycle5}
\end{figure}

\begin{definition}\label{def:barycentricHcycle}Barycentric 1-Cycle.\\
Let $B$ be the set of barycenters on the triangles on a homotopic cycle in a planar CW space $K$.  The edges between each pair of adjacent barycenters in $B$ define a barycentric 1-cycle.
cycle.\qquad\textcolor{blue}{\Squaresteel}
\end{definition}

\begin{theorem}
A complete sequence of paths over the vertexes in a barycentric 1-cycle constructs a barycentric homotopic cycle.
\end{theorem}
\begin{proof}
Let $B$ be the set of barycenters on the triangles in a homotopic cycle and let $h:I\to B$ be a homotopic map over $B$.  From Def.~\ref{def:filledHomotopicCycle}, the sequence of maps $h_0,h_1,\dots,h_{i\mbox{mod}n}$ over adjacent pairs of barycenters $b,b'$ in $B$ ($h_i(0)=b,h_i(1)=b'$) constructs a barycentric homotopic cycle.
\end{proof}

\begin{example}\label{ex:barycentricHomCyc}
A sample planar barycentric homotopic cycle is shown in Fig.~\ref{fig:homotopicCycle5}. \qquad\textcolor{blue}{\Squaresteel}
\end{example}

\section{Rotman Free Group Presentation}
A finite group $G$ is free, provided every element $x\in G$ is
a linear combination of its basis elements (called generators)~\cite[\S 1.4, p. 21]{Munkres1984}.
We write $\mathcal{B}$ to denote a nonempty basis set of generators $\left\{g_1,\dots,g_{\abs{\mathcal{B}}}\right\}$ and $G(\mathcal{B},+)$ to denote the free group with binary operation $+$.

\begin{definition}\label{RotmanPresentation}{{\bf Rotman Presentation}\rm \cite[p.239]{Rotman1965theoryOfGroups}}$\mbox{}$\\
Let $X = \left\{g_1,\dots\right\},\bigtriangleup = \left\{v = \sum kg_i,v\in \mbox{group}G,g_i\in X\right\}$ be a set of generators of members of a nonempty set $X$ and set of relations between members of $G$ and the generators in $X$.  A mapping of the form $\left\{X,\bigtriangleup\right\}\to G$, a free group, is called a presentation of $G$. \qquad\textcolor{blue}{\Squaresteel}
\end{definition}

We write $G(V,+)$ to denote a group $G$ on a nonvoid set $V$ with a binary operation $+$.  For a group $G(V,+)$ presentable as a collection of linear combinations of members of a basis set $\mathcal{B}\subseteq V$, we write $G(\mathcal{B},+)$.

\begin{definition}\label{def:cellComplexFreeGroup}{\bf Free Group Presentation of a Cell Complex}.$\mbox{}$\\
Let $2^K$ be the collection of cell complexes in a CW space $K$, $E\in 2^K$ containing $n$ vertexes, $G(E,+)$ a group on nonvoid set $E$ with binary operation $+$, $\bigtriangleup = \left\{v = \sum kg_i,v\in E,g_i\in E\right\}$ be a set of generators of members in $E$, set of relations between members of $E$ and the generators $\mathcal{B}\subset E$, $g_i\in\mathcal{B},v = h_{i\mbox{mod}\ n}(0)\in K$, $k_i$ the $i^{th}$ integer coeficient $\mbox{mod}n$ in a linear combination $\mathop{\sum}\limits_{i,j}k_ig_j$ of generating elements $g_j = h_j(0)\in \mathcal{B}$.   A free group {\bf presentation} of $G$ is a continuous map $f:2^K\times\bigtriangleup\to 2^K$ defined by
\begin{align*}
f(\mathcal{B},\bigtriangleup) &= \left\{v \assign{\mathop{\sum}\limits_{i,j}k_ig_j\in\bigtriangleup}: v\in E,g_j\in \mathcal{B}, k_i\in\mathbb{Z}\right\}\\
          &= \overbrace{\boldsymbol{G(\left\{ g_1,\dots,g_{_{\abs{\tiny \mathcal{B}}}}\right\}},+).}^{\mbox{\textcolor{blue}{\bf $\mathcal{B}\times\bigtriangleup\mapsto$ free group $\boldsymbol{G}$ presentation of $G(E,+)$}}}\mbox{\qquad\textcolor{blue}{\Squaresteel}}
\end{align*}
\end{definition}

\begin{lemma}~\cite[\S 4, p. 10]{Peters2020BrouwerFPtheorem}
Every 1-cycle in a CW space $K$ has a free group presentation.
\end{lemma}
\begin{proof}
By definition, the vertices in a 1-cycle $\cyc E$ are path-connected in a CW space $K$.  From~\cite{Peters2020BrouwerFPtheorem}, every 1-cycle can represented as a group $G(\cyc E,+)$ with binary (move) operation $+$. Choose a vertex $g = v_0\in\cyc E$ and let basis $\mathcal{B}_{\cyc E} = \left\{g\right\}$.  Every vertex $v_i\in\cyc E, i > 0$ is reachable in $k$ moves from $g\in \mathcal{B}$.  Then we have $v_i = kg = v_{0+k}$. Let $\bigtriangleup = \left\{v = \sum kg,g\in\mathcal{B},\mbox{\&}\ v,g\in \cyc E\right\}$. Then let the map $f:2^K\times\bigtriangleup\to 2^K$ be defined by
\begin{align*}
f(\mathcal{B}_{\cyc E},\bigtriangleup) &= \left\{v \assign{\mathop{\sum}\limits_{i,j}k_ig\in\bigtriangleup}: v\in \cyc E,g\in \mathcal{B}_{\cyc E}, k_i\in\mathbb{Z}\right\}\\
          &= \overbrace{\boldsymbol{G(\left\{g\right\}},+).}^{\mbox{\textcolor{blue}{\bf $\boldsymbol{\mathcal{B}_{\cyc E}\times\bigtriangleup}\mapsto$ free group $\boldsymbol{G}$ presentation of $G(\cyc E,+)$}}}
\end{align*}
\end{proof}

\section{Main Results}

\noindent This section introduces main results for homotopic nerves.\\
\vspace{3mm}

\begin{lemma}~\cite{PetersVergili2021proximalHomotopy}\label{theorem:hCycPresentation}
Every homotopic cycle in a CW space has a free group presentation.
\end{lemma}
\begin{proof}
Let $E$ be a set of 1-cycle vertices in the domain and range of the maps $h:I\to E$ in a homotopic cycle $\hcyc E$ in a CW space $K$, $G(\hcyc E,+)$ a group on the vertexes of $\hcyc E$ with move operation $+$, and let
\begin{align*}
\mathcal{B} &\subset \hcyc E,\\
\bigtriangleup &= \left\{v = \sum kh_i(0),v,h_i(0)\in \hcyc E\right\}.
\end{align*} 
be the basis $\mathcal{B}$ set of members of $\hcyc E$ and set of relations $\bigtriangleup$ between members of $\hcyc E$ and the generators in basis $\mathcal{B}$.  Then the free group presentation of $G$ is a mapping $f:2^K\times\bigtriangleup\to 2^K$ defined by
\begin{align*}
f(\mathcal{B},\bigtriangleup) &= \left\{v \assign{\mathop{\sum}\limits_{i,j}k_ih_j(0)\in\bigtriangleup}: v\in \hcyc E,h_j\in \mathcal{B}, k_i\in\mathbb{Z}\right\}\\
          &= \overbrace{\boldsymbol{G(\left\{h_1,\dots,h_{_{\abs{\tiny \mathcal{B}}}}\right\}},+).}^{\mbox{\textcolor{blue}{\bf $\boldsymbol{\mathcal{B}\times\bigtriangleup}\mapsto$ free group presentation of $G(\hcyc E,+)$}}}
\end{align*}
\end{proof}

\begin{example}
From Lemma~\ref{theorem:hCycPresentation}, the barycentric homotopic cycle $\bhcyc E'$ in Example~\ref{ex:barycentricHomCyc} has a free group presentation.
\qquad\textcolor{blue}{\Squaresteel}
\end{example}


\begin{theorem}
Every homotopic nerve in a CW space has a free group presentation.
\end{theorem}
\begin{proof}
Symmetric with the proof of Lemma~\ref{theorem:hCycPresentation}, after replacing homotopic cycle $\hcyc E$ with homotopic nerve $\hNrv E$ and replacing generator $g_j$ with $h_j(0)$ in the basis for the free group presentation of the homotopic nerve in a CW space.
\end{proof}



\begin{figure}[!ht]
\centering
\begin{pspicture}
(-3.5,-1.0)(4.0,4.0)
\centering
\psframe[linewidth=0.75pt,linearc=0.25,cornersize=absolute,linecolor=blue](-3.8,-1.0)(4.8,4.0)
\psline[linewidth=0.5pt,linecolor=blue,arrowscale=1.0]{-v}%
(-1.0,2)(-0.5,3)
\psline[linewidth=0.5pt,linecolor=blue,arrowscale=1.0]{-v}%
(-0.5,3)(0.5,3.5)
\psline[linewidth=0.5pt,linecolor=blue,arrowscale=1.0]{-v}%
(0.5,3.5)(1.5,3.0)
\psline*[linecolor=gray!20]
(1.5,3.0)(2.0,2.0)(0.5,1.5)
\psline[linewidth=0.5pt,linecolor=magenta,arrowscale=1.0]{-v}%
(1.5,2.25)(1.5,3.0)
\psline[linewidth=0.5pt,linecolor=magenta,arrowscale=1.0]{-v}%
(-0.3,0.85)(-0.5,1.5)
\psline[linewidth=0.5pt,linecolor=magenta,arrowscale=1.0]{-v}%
(-0.5,1.5)(-0.4,2.25)
\psline[linewidth=0.5pt,linecolor=magenta,arrowscale=1.0]{-v}%
(-0.4,2.25)(0.1,2.95)
\psline[linewidth=0.5pt,linecolor=magenta,arrowscale=1.0]{-v}%
(0.1,2.95)(0.85,2.95)
\psline[linewidth=0.5pt,linecolor=magenta,arrowscale=1.0]{-v}%
(0.85,2.95)(1.5,2.25)
\psline[linewidth=0.5pt,linecolor=magenta,arrowscale=1.0]{-v}%
(1.5,0.75)(0.9,0.25)
\psline[linewidth=0.5pt,linecolor=magenta,arrowscale=1.0]{-v}%
(1.5,2.25)(1.5,0.85)
\psline[linewidth=0.5pt,linecolor=magenta,arrowscale=1.0]{-v}%
(0.9,0.25)(-0.3,0.85)
\psdots[dotstyle=o,dotsize=2.5pt,linewidth=1.2pt,linecolor=black,fillcolor=green!80]
(-0.3,0.85)(-0.5,1.5)(-0.4,2.25)(0.1,2.95)(0.85,2.95)%
(1.5,2.25)(1.5,0.75)(0.9,0.25)
\psline[linewidth=0.5pt,linecolor=blue,arrowscale=1.0]{-v}%
(1.5,3.0)(2.0,2.0)
\psline[linewidth=0.5pt,linecolor=blue,arrowscale=1.0]{-v}%
(1.5,3.0)(0.5,1.5)
\psline[linewidth=0.5pt,linecolor=blue,arrowscale=1.0]{-v}%
(0.5,1.5)(2.0,2.0)
\psline[linewidth=0.5pt,linecolor=blue,arrowscale=1.0]{-v}%
(0.5,1.5)(2.0,1.0)
\psline[linewidth=0.5pt,linecolor=blue,arrowscale=1.0]{-v}%
(1.5,0.0)(0.5,1.5)
\psline[linewidth=0.5pt,linecolor=blue,arrowscale=1.0]{-v}%
(0.5,1.5)(0.5,-0.5)
\psline[linewidth=0.5pt,linecolor=blue,arrowscale=1.0]{-v}%
(1.5,0.0)(0.5,1.5)
\psline[linewidth=0.5pt,linecolor=blue,arrowscale=1.0]{-v}%
(0.5,1.5)(-1.0,1.0)
\psline[linewidth=0.5pt,linecolor=blue,arrowscale=1.0]{-v}%
(-1.0,2.0)(0.5,1.5)
\psline[linewidth=0.5pt,linecolor=blue,arrowscale=1.0]{-v}%
(0.5,1.5)(-0.5,3.0)
\psline[linewidth=0.5pt,linecolor=blue,arrowscale=1.0]{-v}%
(0.5,1.5)(0.5,3.5)
\psline[linewidth=0.5pt,linecolor=blue,arrowscale=1.0]{-v}%
(-0.5,0.0)(0.5,1.5)
\psdots[dotstyle=o,dotsize=2.8pt,linewidth=1.2pt,linecolor=black,fillcolor=green!80](0.5,1.5)
\psline[linewidth=0.5pt,linecolor=blue,arrowscale=1.0]{-v}%
(2.0,1.0)(1.5,0.0)
\psline[linewidth=0.5pt,linecolor=blue,arrowscale=1.0]{-v}%
(1.5,0.0)(0.5,-0.5)
\psline[linewidth=0.5pt,linecolor=blue,arrowscale=1.0]{-v}%
(-0.5,0.0)(-1.0,1.0)
\psline[linewidth=0.5pt,linecolor=blue,arrowscale=1.0]{-v}%
(-1.0,1.0)(-1.0,2)
\psdots[dotstyle=o,dotsize=2.5pt,linewidth=1.2pt,linecolor=black,fillcolor=red!80]
(-1.0,2)(-0.5,3)(0.5,3.5)(1.5,3.0)(2.0,2.0)(2.0,1.0)(1.5,0.0)%
(0.5,-0.5)(-0.5,0.0)(-1.0,1.0)(-1.0,1.0)(-1.0,2)
\psdots[dotstyle=o,dotsize=2.8pt,linewidth=1.2pt,linecolor=black,fillcolor=green!80]
(1.5,3.0)
\rput(-0.7,2.0){\footnotesize $\boldsymbol{v_1}$}
\rput(-0.7,0.9){\footnotesize $\boldsymbol{v_0}$}
\psdots[dotstyle=o,dotsize=1.5pt,linewidth=1.2pt,linecolor=black,fillcolor=red!80]
(2.0,1.7)(2.0,1.5)(2.0,1.3)
\psdots[dotstyle=o,dotsize=1.5pt,linewidth=1.2pt,linecolor=black,fillcolor=red!80]
(0.2,-0.4)(0.0,-0.3)(-0.2,-0.2)
\rput(-3.2,3.8){\footnotesize $\boldsymbol{K}$}
\rput(3.3,0.9){\footnotesize %
$\boldsymbol{h_{i}(0) = h_{i-1}(1)}$} 
\rput(2.1,0){\footnotesize $\boldsymbol{h_{i}(1)}$}
\rput(-2.5,0.9){\footnotesize %
$\boldsymbol{h_{0}(0) =h_{(n-1)[n]}(1)}$} 
\rput(-2.55,0.6){\footnotesize $\boldsymbol{=\ v_0}$}%
\rput(-2.2,2.1){\footnotesize $
\boldsymbol{h_1(0)\ = h_0(1)}$}
\rput(0.9,3.7){\footnotesize $
\boldsymbol{h_3(0) = v_3}$}
\rput(-1.1,-0.2){\footnotesize $
\boldsymbol{h_{(n-1)[n]}(0)}$}
\rput(3.0,3.2){\colorbox{gray!20}{$
\boldsymbol{\hbr A}$}}
\psline[linewidth=1.2pt,linestyle=dotted,linecolor=blue,arrowscale=0.5]{-v}%
(3.0,3.0)(1.5,2.5)
\rput(-2.5,3.2){\colorbox{gray!20}{$
\boldsymbol{\hcyc E'}$}}
\psline[linewidth=1.2pt,linestyle=dotted,linecolor=blue,arrowscale=0.5]{-v}%
(-2.5,3.0)(-0.35,2.5)
\rput(-3.0,1.5){\colorbox{gray!20}{$
\boldsymbol{\hcyc E}$}}
\psline[linewidth=1.2pt,linestyle=dotted,linecolor=blue,arrowscale=0.5]{-v}%
(-2.6,1.5)(-1.0,1.5)
\rput(1.95,3.0){\footnotesize $
\boldsymbol{h_4(0)}$}
\rput(2.4,2.0){\footnotesize $
\boldsymbol{h_5(0)}$}
\rput(0.8,1.8){\footnotesize $
\boldsymbol{c}$}
\end{pspicture}
\caption[]{Homotopic vortex $\boldsymbol{\vor E}$.}
\label{fig:homotopicCycle8}
\end{figure} 

\begin{remark}
A collection of \emph{                                  \bf Nested 1-cycles} provides the basic structure of a vortex.  Cycles are \emph{nested}, provided each appear, one inside the other.  For example, a pair of 1-cycles $\cyc E,\cyc E'$ are nested, provided $\cyc E\subset \Int(\cyc E')$. In this work, intersecting, nested cycles are possible, {\em i.e.},
\[
\cyc E\subset \Int(\cyc E')\ \mbox{and, possibly,}\ \cyc E\cap \cyc E'\neq \emptyset.
\]
That is, each pair of nested cycles in a vortex can have at least one vertex in common.  To solve the problem of structuring a vortex $\vor E$ containing non-intersecting nested cycles so that $\vor E$ can be mapped to a homotopic vortex, bridge edges are introduced (see Def.~\ref{def:hVortex}).
\qquad\textcolor{blue}{\Squaresteel}
\end{remark}

\begin{definition}
Given a pair of nested 1-cycles $\cyc E,\cyc E'$, a {\bf bridge edge} $A$ (denoted by $\br A$) is an edge with one vertex on $\cyc E$ and one vertex on $\cyc E'$, or vice versa.
\qquad\textcolor{blue}{\Squaresteel}
\end{definition}

A homotopic bridge is a path between edge vertices of a bridge edge.

\begin{definition}
Given a pair of nested, homotopic 1-cycles $\hcyc E,\hcyc E'$, a {\bf homotopic bridge} $A$ (denoted by $\hbr A$) is a homotopic map $h$ such that $h(0) = v\in\hcyc E$ and $h(1) = v'\in\hcyc E'$, or vice versa.
\qquad\textcolor{blue}{\Squaresteel}
\end{definition}

\begin{example}
A sample homotopic bridge $\hbr A$ between a pair of homotopic cycles is shown in Fig.~\ref{fig:homotopicCycle8}.
\qquad\textcolor{blue}{\Squaresteel}
\end{example}

\begin{definition}
A {\bf vortex cell complex} $E$ (denoted by $\vor E$) is a collection of nested 1-cycles with at least one bridge edge between a pair of vertexes on the cycles in the vortex, or the 1-cycles in $\vor E$ have nonvoid intersection.
\qquad\textcolor{blue}{\Squaresteel}
\end{definition}

By replacing each of the edges in a vortex cell complex $\vor E$ with a homotopic map and replacing every bridge edge in $\vor E$
with a homotopic bridge, we obtain a homotopic vortex.

\begin{definition}\label{def:hVortex}
A {\bf homotopic vortex} $E$ (denoted by $\hvor E$) is a collection of nested homotopic 1-cycles with nonvoid intersection or with at least one homotopic bridge between each pair of non-intersecting cycle vertexes in the vortex.
\qquad\textcolor{blue}{\Squaresteel}
\end{definition}

\begin{example}
A sample homotopic vortex $\hvor E$ with a bridge edge is shown in Fig.~\ref{fig:homotopicCycle8}.
\qquad\textcolor{blue}{\Squaresteel}
\end{example}

\begin{theorem}\label{thm:hvor}
Every homotopic vortex has a free group presentation.
\end{theorem}
\begin{proof}
Symmetric with the proof of Lemma~\ref{theorem:hCycPresentation}, after replacing homotopic cycle $\hcyc E$ with homotopic vortex $\hvor E$ in group $G(\hvor E,+)$ and replacing generator $g_j$ with $h_j(0)$ in basis $\mathcal{B}$, the set of generator vertexes either on the end points of the homotopic bridges or in the intersection of the $\hvor E$ homotopic cycles.  Then introduce
\[
\bigtriangleup = \left\{v = \sum kh_i(0),v\in E,h_i(0)\in \mathcal{B}\subset E\right\}
\]
with $E$ equal to the set of vertexes in the domain and range of the maps $h:I\to E$ in $\hvor E$ in a CW space $K$ and introduce a continuous mapping $f:2^K\times\bigtriangleup\to 2^K$ defined by
\begin{align*}
f(\mathcal{B},\bigtriangleup) &= \left\{v \assign{\mathop{\sum}\limits_{i,j}k_ih_j(0)\in\bigtriangleup}: v\in E,h_j\in \mathcal{B}, k_i\in\mathbb{Z}\right\}\\
          &= \overbrace{\boldsymbol{G(\mathcal{B},+).}}^{\mbox{\textcolor{blue}{\bf $\boldsymbol{\mathcal{B}\times\bigtriangleup}\mapsto$ free group presentation of $G(\hvor E,+)$}}}
\end{align*}
\end{proof}

\begin{figure}[!ht]
\centering
\begin{pspicture}
(-3.5,-1.0)(4.0,5.0)
\centering
\psframe[linewidth=0.75pt,linearc=0.25,cornersize=absolute,linecolor=blue](-4.8,-1.0)(4.8,5.0)
\psline*[linecolor=gray!20]
(0.5,1.5)(0.5,3.5)(1.5,3.0)
\psline[linewidth=0.5pt,linecolor=blue,arrowscale=1.0]{-v}%
(-1.0,2)(-0.5,3)
\psline[linewidth=0.5pt,linecolor=blue,arrowscale=1.0]{-v}%
(-0.5,3)(0.5,3.5)
\psline[linewidth=0.5pt,linecolor=blue,arrowscale=1.0]{-v}%
(0.5,3.5)(1.5,3.0)
\psline*[linecolor=gray!20]
(1.5,3.0)(2.0,2.0)(0.5,1.5)
\psline[linewidth=0.5pt,linecolor=magenta,arrowscale=1.0]{-v}%
(-0.3,0.85)(-0.5,1.5)
\psline[linewidth=0.5pt,linecolor=magenta,arrowscale=1.0]{-v}%
(-0.5,1.5)(-0.4,2.25)
\psline[linewidth=0.5pt,linecolor=magenta,arrowscale=1.0]{-v}%
(-0.4,2.25)(0.1,2.95)
\psline[linewidth=0.5pt,linecolor=magenta,arrowscale=1.0]{-v}%
(0.1,2.95)(0.85,2.95)
\psline[linewidth=0.5pt,linecolor=magenta,arrowscale=1.0]{-v}%
(0.85,2.95)(1.5,3.0)
\psline[linewidth=0.5pt,linecolor=magenta,arrowscale=1.0]{-v}%
(1.5,3.0)(1.5,2.25)
\psline[linewidth=0.5pt,linecolor=magenta,arrowscale=1.0]{-v}%
(1.5,0.75)(0.9,0.25)
\psline[linewidth=0.5pt,linecolor=magenta,arrowscale=1.0]{-v}%
(1.5,2.25)(1.5,0.85)
\psline[linewidth=0.5pt,linecolor=magenta,arrowscale=1.0]{-v}%
(0.9,0.25)(-0.3,0.85)
\psdots[dotstyle=o,dotsize=2.5pt,linewidth=1.2pt,linecolor=black,fillcolor=green!80]
(-0.3,0.85)(-0.5,1.5)(-0.4,2.25)(0.1,2.95)(0.85,2.95)%
(1.5,2.25)(1.5,0.75)(0.9,0.25)
\psline[linewidth=0.5pt,linecolor=blue,arrowscale=1.0]{-v}%
(1.5,3.0)(2.0,2.0)
\psline[linewidth=0.5pt,linecolor=blue,arrowscale=1.0]{-v}%
(1.5,3.0)(0.5,1.5)
\psline[linewidth=0.5pt,linecolor=blue,arrowscale=1.0]{-v}%
(0.5,1.5)(2.0,2.0)
\psline[linewidth=0.5pt,linecolor=blue,arrowscale=1.0]{-v}%
(0.5,1.5)(2.0,1.0)
\psline[linewidth=0.5pt,linecolor=blue,arrowscale=1.0]{-v}%
(1.5,0.0)(0.5,1.5)
\psline[linewidth=0.5pt,linecolor=blue,arrowscale=1.0]{-v}%
(0.5,1.5)(0.5,-0.5)
\psline[linewidth=0.5pt,linecolor=blue,arrowscale=1.0]{-v}%
(1.5,0.0)(0.5,1.5)
\psline[linewidth=0.5pt,linecolor=blue,arrowscale=1.0]{-v}%
(0.5,1.5)(-1.0,1.0)
\psline[linewidth=0.5pt,linecolor=blue,arrowscale=1.0]{-v}%
(-1.0,2.0)(0.5,1.5)
\psline[linewidth=0.5pt,linecolor=blue,arrowscale=1.0]{-v}%
(0.5,1.5)(-0.5,3.0)
\psline[linewidth=0.5pt,linecolor=blue,arrowscale=1.0]{-v}%
(0.5,1.5)(0.5,3.5)
\psline[linewidth=0.5pt,linecolor=blue,arrowscale=1.0]{-v}%
(-0.5,0.0)(0.5,1.5)
\psdots[dotstyle=o,dotsize=2.8pt,linewidth=1.2pt,linecolor=black,fillcolor=green!80](0.5,1.5)
\psline[linewidth=0.5pt,linecolor=blue,arrowscale=1.0]{-v}%
(2.0,1.0)(1.5,0.0)
\psline[linewidth=0.5pt,linecolor=blue,arrowscale=1.0]{-v}%
(1.5,0.0)(0.5,-0.5)
\psline[linewidth=0.5pt,linecolor=blue,arrowscale=1.0]{-v}%
(-0.5,0.0)(-1.0,1.0)
\psline[linewidth=0.5pt,linecolor=blue,arrowscale=1.0]{-v}%
(-1.0,1.0)(-1.0,2)
\psdots[dotstyle=o,dotsize=2.5pt,linewidth=1.2pt,linecolor=black,fillcolor=red!80]
(-1.0,2)(-0.5,3)(0.5,3.5)(1.5,3.0)(2.0,2.0)(2.0,1.0)(1.5,0.0)%
(0.5,-0.5)(-0.5,0.0)(-1.0,1.0)(-1.0,1.0)(-1.0,2)
\psdots[dotstyle=o,dotsize=2.8pt,linewidth=1.2pt,linecolor=black,fillcolor=green!80]
(1.5,3.0)
\rput(-0.7,2.0){\footnotesize $\boldsymbol{v_1}$}
\rput(-0.7,0.9){\footnotesize $\boldsymbol{v_0}$}
\psdots[dotstyle=o,dotsize=1.5pt,linewidth=1.2pt,linecolor=black,fillcolor=red!80]
(2.0,1.7)(2.0,1.5)(2.0,1.3)
\psdots[dotstyle=o,dotsize=1.5pt,linewidth=1.2pt,linecolor=black,fillcolor=red!80]
(0.2,-0.4)(0.0,-0.3)(-0.2,-0.2)
\rput(-4.2,4.8){\footnotesize $\boldsymbol{K}$}
\rput(3.1,0.9){\footnotesize %
$\boldsymbol{h_{i}(0) = h_{i-1}(1)}$} 
\rput(2.1,0){\footnotesize $\boldsymbol{h_{i}(1)}$}
\rput(-2.5,0.9){\footnotesize %
$\boldsymbol{h_{0}(0) =h_{(n-1)[n]}(1)}$} 
\rput(-2.55,0.6){\footnotesize $\boldsymbol{=\ v_0}$}%
\rput(-2.2,2.1){\footnotesize $
\boldsymbol{h_1(0)\ = h_0(1)}$}
\rput(-1.1,-0.2){\footnotesize $
\boldsymbol{h_{(n-1)[n]}(0)}$}
\rput(2.0,4.5){\colorbox{gray!20}{$
\boldsymbol{\hcyc E\cap\bhcyc E'=h_4(0)}$}}
\psline[linewidth=1.2pt,linestyle=dotted,linecolor=blue,arrowscale=0.5]{-v}%
(3.0,4.5)(2.0,3.3)
\rput(-2.5,3.2){\colorbox{gray!20}{$
\boldsymbol{\bhcyc E'}$}}
\psline[linewidth=1.2pt,linestyle=dotted,linecolor=blue,arrowscale=0.5]{-v}%
(-2.5,3.0)(-0.35,2.5)
\rput(-3.0,1.5){\colorbox{gray!20}{$
\boldsymbol{\hcyc E}$}}
\psline[linewidth=1.2pt,linestyle=dotted,linecolor=blue,arrowscale=0.5]{-v}%
(-2.6,1.5)(-1.0,1.5)
\rput(1.95,3.0){\footnotesize $
\boldsymbol{h_4(0)}$}
\rput(2.4,2.0){\footnotesize $
\boldsymbol{h_5(0)}$}
\rput(0.8,1.8){\footnotesize $
\boldsymbol{c}$}
\end{pspicture}
\caption[]{Homotopic vortex nerve $\boldsymbol{\vorNrv E}$.}
\label{fig:homotopicCycle13}
\end{figure}

\begin{definition}
A {\bf vortex nerve} $E$ (denoted by $\vorNrv E$) is a collection of nested 1-cycles that have nonempty intersection.
\qquad\textcolor{blue}{\Squaresteel}
\end{definition}

By replacing each of the edges in a vortex nerve $\vorNrv E$ homotopic paths, we obtain a homotopic vortex nerve.

\begin{definition}
A {\bf homotopic vortex nerve} $E$ (denoted by $\hvorNrv E$) is a collection of homotopic vortexes that have nonempty intersection.
\qquad\textcolor{blue}{\Squaresteel}
\end{definition}

\begin{example}
A sample homotopic vortex nerve $\hvorNrv E$ containing a single vortex is shown in Fig.~\ref{fig:homotopicCycle13}.  This is a homotopic vortex nerve, which results from the intersecting of a pair of nested homotopic cycles $\hcyc E,bhcyc E'$, {\em i.e.},
\[
\hcyc E\cap\bhcyc E'=h_4(0).
\]
This is also an example of an Alexandroff nerve, since the homotopic cycles in $\hvorNrv E$ have nonvoid intersection. The minimal basis $\mathcal{B} = \left\{h_4(0)\right\}$ in the free group presentation of this nerve.
\qquad\textcolor{blue}{\Squaresteel}
\end{example}

\begin{figure}[!ht]
\centering
\begin{pspicture}
(-1.5,-0.5)(4.0,4.0)
\psframe[linewidth=0.75pt,linearc=0.25,cornersize=absolute,linecolor=blue](-2.5,-0.5)(4.5,3.0)
\psline*[linestyle=solid,linecolor=green!30]%
(0,0)(1,1)(0,2)(-1,1.5)(-1,0.5)(0,0)
\psline[linewidth=0.5pt,linecolor=blue,arrowscale=1.0]{-v}%
(0,0)(1,1)
\psline[linewidth=0.5pt,linecolor=blue,arrowscale=1.0]{-v}%
(1,1)(0,2)
\psline[linewidth=0.5pt,linecolor=blue,arrowscale=1.0]{-v}%
(0,2)(-1,1.5)
\psline[linewidth=0.5pt,linecolor=blue,arrowscale=1.0]{-v}%
(-1,1.5)(-1,0.5)
\psline[linewidth=0.5pt,linecolor=blue,arrowscale=1.0]{-v}%
(-1,0.5)(0,0)
\psline[linewidth=0.5pt,linecolor=blue,arrowscale=1.0]{-v}%
(0,0)(1,1)
\psline*[linestyle=solid,linecolor=orange!50]%
(1,1)(-.55,1.25)(-.55,0.75)(0,0.5)(1,1)
\psline[linewidth=0.5pt,linecolor=blue,arrowscale=1.0]{-v}%
(1,1)(-.55,1.25)
\psline[linewidth=0.5pt,linecolor=blue,arrowscale=1.0]{-v}%
(-.55,1.25)(-.55,0.75)
\psline[linewidth=0.5pt,linecolor=blue,arrowscale=1.0]{-v}%
(-.55,0.75)(0,0.5)
\psline[linewidth=0.5pt,linecolor=blue,arrowscale=1.0]{-v}%
(0,0.5)(1,1)
\psdots[dotstyle=o,dotsize=2.2pt,linewidth=1.2pt,linecolor=black,fillcolor=yellow!80]%
(0,0.5)(1,1)(-.55,1.25)(-.55,0.75)(0,0.5)
(0,2)(-1,1.5)(-1,0.5)(0,0) 
\psline*[linestyle=solid,linecolor=green!30]%
(1,1)(2.0,2.0)(3.0,1.5)(3.0,0.5)(2.0,0.0)(1,1)
\psline[linewidth=0.5pt,linecolor=blue,arrowscale=1.0]{-v}%
(1,1)(2.0,2.0)
\psline[linewidth=0.5pt,linecolor=blue,arrowscale=1.0]{-v}%
(2.0,2.0)(3.0,1.5)
\psline[linewidth=0.5pt,linecolor=blue,arrowscale=1.0]{-v}%
(3.0,1.5)(3.0,0.5)
\psline[linewidth=0.5pt,linecolor=blue,arrowscale=1.0]{-v}%
(3.0,0.5)(2.0,0.0)
\psline[linewidth=0.5pt,linecolor=blue,arrowscale=1.0]{-v}%
(2.0,0.0)(1,1)
\psline*[linestyle=solid,linecolor=orange!50]%
(1,1)(2.55,1.25)(2.55,0.75)(2.0,0.5)(1,1)
\psline[linewidth=0.5pt,linecolor=blue,arrowscale=1.0]{-v}%
(1,1)(2.55,1.25)
\psline[linewidth=0.5pt,linecolor=blue,arrowscale=1.0]{-v}%
(2.55,1.25)(2.55,0.75)
\psline[linewidth=0.5pt,linecolor=blue,arrowscale=1.0]{-v}%
(2.55,0.75)(2.0,0.5)
\psline[linewidth=0.5pt,linecolor=blue,arrowscale=1.0]{-v}%
(2.0,0.5)(1,1)
\psline[linestyle=solid,linecolor=black]%
(1,1)(2.55,1.25)(2.55,0.75)(2.0,0.5)(1,1)
\psline[linestyle=solid,linecolor=black]%
(1,1)(2.55,1.25)(2.55,0.75)(2.0,0.5)(1,1)
\psdots[dotstyle=o,dotsize=2.2pt,linewidth=1.2pt,linecolor=black,fillcolor=yellow!80]%
(1,1)(2.0,0.0)(3.0,0.5)(3.0,1.5)(2.0,2.0)
(2.55,1.25)(2.55,0.75)(2.0,0.5)
\psdots[dotstyle=o,dotsize=3.5pt,linewidth=1.2pt,linecolor=black,fillcolor=red!80]
(1.0,1)
\rput(-0.8,1.85){\footnotesize $\boldsymbol{\hcyc A}$}
\rput(-0.25,1.5){\footnotesize $\boldsymbol{\hcyc B}$}
\rput(2.85,1.85){\footnotesize $\boldsymbol{\hcyc A'}$}
\rput(2.25,1.5){\footnotesize $\boldsymbol{\hcyc B'}$}
\rput(1.,1.25){\footnotesize $\boldsymbol{v_0}$}
\rput(-1.8,2.8){\footnotesize $\boldsymbol{K}$}
\rput(1.0,2.5){\colorbox{gray!20}
{$\boldsymbol{\hvorNrv E}$}}
\rput(-1.8,1.0){\colorbox{gray!20}{$
\boldsymbol{\hvor E}$}}
\rput(3.8,1.0){\colorbox{gray!20}{$
\boldsymbol{\hvor E'}$}}
\end{pspicture}
\caption[]{Vigolo Hawaiian earrings pair $\boldsymbol{\to}$ homotopic vortex nerve $\boldsymbol{\hvorNrv E}$}
\label{fig:HawaiianVortexNerve}
\end{figure}

\begin{example}
A second homotopic vortex nerve $\hvorNrv E$ is shown in Fig.~\ref{fig:HawaiianVortexNerve}.  This is a homotopic vortex nerve, which results from an intersecting of a pair of vortexes $\hvor E,\hvor E'$, {\em i.e.},
\[
\hvor E \cap \hvor E'= \boldsymbol{v_0}.
\]
This nerve has been derived from a homotopic version of a Vigolo Hawaiian earring~\cite{Vigolo2018HawaiianEarrings}, which is a wide ribbon for Vigolo.  The minimal basis $\mathcal{B} = \left\{v_0\right\}$ in the free group presentation of this nerve.
\qquad\textcolor{blue}{\Squaresteel}
\end{example}

\begin{theorem}
Every homotopic vortex nerve has a free group presentation.
\end{theorem}
\begin{proof}
Symmetric with the proof of Theorem~\ref{thm:hvor}, after replacing homotopic vortex $\hvor E$ with homotopic nerve $\hvorNrv E$ and replacing the homotopic vortex basis with basis $\mathcal{B}$, the set of generator vertexes either on the end points of the homotopic bridges between the vortex cycles in $\hvorNrv E$ or in the intersection of the $\hvor E$ homotopic cycles in $\hvorNrv E$. 
\end{proof}

\begin{definition}\label{def:EHnerve} {\bf Edelsbrunner-Harer Nerve} {\rm \cite[\S III.2,p. 59]{Edelsbrunner1999}}.\\
Let $F$ be a finite collection of sets.  A {\bf nerve} consists of all nonempty sets $X$ in $F$ that have nonvoid intersection, {\em i.e.},
\[
\Nrv F = \left\{X\subseteq F: \bigcap X\neq \emptyset\right\}. \mbox{\qquad\textcolor{blue}{\Squaresteel}}
\]
\end{definition}

\begin{remark}
Def.~\ref{def:EHnerve} is a formal definition of the original view of an Alexandrov nerve complex~\cite{Alexandroff1932elementaryConcepts,Alexandroff1926MAnnNerfTheorem}.
\qquad\textcolor{blue}{\Squaresteel} 
\end{remark}

\begin{theorem}\label{theorem:EHnerve}{\rm~\cite[\S III.2,p. 59]{Edelsbrunner1999}}.\\
Let $F$ be a finite collection of closed, convex sets in Euclidean space.  Then the nerve of $F$ and union of the sets in $F$ have the same homotopy type.
\end{theorem}

\noindent Let $i_X,i_Y$ be identity maps, {\em i.e.}, 
\begin{align*}
i_X(x) &= x,\ \mbox{for all}\ x\in X.\\
i_Y(y) &= y,\ \mbox{for all}\ y\in Y.
\end{align*}
Recall that two spaces $X$ and $Y$ have the same {\bf homotopy type}~\cite[\S 19,p. 108]{Munkres1984} ({\em i.e.}, spaces $X$ and $Y$ are {\bf homotopy equivalent} denoted by $X\equiv Y$ using Hilton's notation~\cite{Hilton1952homotopy}), provided the maps 
\[
f:X\to Y\ \mbox{and}\ g:Y\to X 
\]
are both homotopic such that 
\[
f\circ g = i_X\ \mbox{and}\ g\circ f = i_Y.
\]

\begin{theorem}
Let $F$ be a vortex nerve that is a finite collection of closed, convex sets in Euclidean space.  Then the nerve of $F$ and union of the sets in $F$ have the same homotopy type.
\end{theorem}
\begin{proof}
There are a number of cases to consider, namely,
\begin{compactenum}[{\bf Case}.1$^o$]
\item {\bf 1-cycle nerve}: Let $\Nrv F$ be a finite collection of closed, convex 1-cycles in Euclidean space, {\em i.e.}, 
\[
\Nrv Cyc F = \left\{\cyc E\subseteq F: \bigcap \cyc E\neq \emptyset\right\}.
\] 
From Theorem~\ref{theorem:EHnerve}, $\Nrv Cyc F$ and the union of the 1-cycles in $F$ have the homotopy type. 
\item {\bf homotopic 1-cycle nerve}: After replacing $\cyc E$ with $\hcyc E$, the proof for this case is symmetric with {\bf Case}.1$^o$.
\item {\bf vortex nerve}: Let $\Nrv F$ be a finite collection of closed, convex, nested, intersecting vortexes in Euclidean space.  After replacing $\cyc E$ with $\vor E$, the proof for this case is symmetric with {\bf Case}.1$^o$, {\em i.e.}, 
\[
\Nrv Vor F = \left\{\vor E\subseteq F: \bigcap \vor E\neq \emptyset\right\}.
\] 
From Theorem~\ref{theorem:EHnerve}, $\Nrv Vor F$ and the union of the vortexes $\vor E$ in $F$ have the homotopy type. 
\item {\bf homotopic vortex nerve}: After replacing $\vor E$ with $\hvor E$ in $\Nrv Vor F$ in {\bf Case}.3$^o$, the proof for this case is symmetric with {\bf Case}.3$^o$.
\end{compactenum}
\end{proof}

\section{Application}
An application of homotopic vortex nerves is given in term of the persistence of descriptively proximal video frame shapes. The motivation for considering free group presentations of such nerves is that we can then describe frame shapes in terms of the Betti numbers of the vortex nerves that lie within the interior of the frame shapes.  A Betti number is a count of the number generators in a free group~\cite[\S 4,p. 24]{Munkres1984}. Frame shapes are, for example, descriptively close, provided the difference between the Betti numbers of the free group presentations of the corresponding homotopic vortex nerves, is close.  Determining the persistence of frame shapes then reduces to tracking the appearance, disappearance and possible reappearance of the shapes in terms of their recurring Betti numbers.  For more about this, see~\cite{Peters2020AMSBullRibbonComplexes}.
\qquad\textcolor{blue}{\Squaresteel}

\begin{appendix}\label{appDetails}
\section{Cell Complexes}\label{ap:CW}
A planar Whitehead cell complex $K$~\cite{Peters2020CGTPhysics} (usually called a CW complex) is a collection of n-dimensional minimal cells $e^n_{\alpha}, n\in \{0,1,2\}$, {\em i.e.},
\[
K = \left\{e^n_{\alpha}\subset \mathbb{R}^2: n\in \{0,1,2\}\right\}.
\]
in the Euclidean plane $\pi$.

\begin{definition}
A cell subcomplex $\sh E\assign\left\{e^n_{\alpha}\right\}\in 2^K$ ({\bf shape complex}) is a {\bf closed subcomplex}, provided the subcomplex includes both a nonempty interior (denoted by $\Int(e^n_{\alpha})$) and its boundary (denoted by $\bdy(e^n_{\alpha})$). In effect, $\sh E$ is closed, provided
\[
\sh E = \Int(\sh E)\cup \bdy(\sh E)\ \mbox{({\bf Closed subcomplex})}.
\]
\end{definition}

\noindent Let $2^{\pi}$ be the collection of all subsets in the Euclidean plane $\pi$.  In the plane, a Whitehead {\bf Closure-finite Weak (CW)} cell complex $K\in 2^{\pi}$ has two properties, namely,
\begin{description}
\item [{\bf C}] A cell complex $K$ is {\bf closure-finite}, provided each cell $e^n_{\alpha}\in K$ is contained in a finite subcomplex of $K$.  In addition,
each cell $e^n_{\alpha}\in K$ has a finite number of immediate faces.  One cell $e^n_{\alpha}$ is an {\bf immediate face} of another cell $e^m_{\alpha}$, provided $e^n_{\alpha}\cap e^m_{\alpha}\neq \emptyset$~\cite{Switzer2002CWcomplex} (also called a {\bf common face}).
\item [{\bf W}] The plane $\pi$ has a {\bf weak topology} induced by cell complex $K$, {\em i.e.}, a subset $S\in 2^{\pi}$ is closed, if and only if $S\cap e^n_{\alpha}$ is also closed in $e^n_{\alpha}$ for each $n, \alpha$~\cite[\S 5.3, p. 65]{Switzer2002CWcomplex}.
\end{description}

A collection $K\in 2^{\pi}$ is called a {\bf CW complex}, provided it has the closure-finite property and $\pi$ has the weak topology property induced by $K$.

Minimal cell planar complexes are given in Table~\ref{tab:skeleton}.

\begin{table}[!ht]\scriptsize
\caption{Minimal Planar Cell Complexes}
\label{tab:skeleton}
\begin{tabular}{|c|c|c|c|}
    \hline
    Minimal Complex & Cell $e^n: n\in \{0,1,2\}$  & Planar Geometry & Interior\\
    \hline
    \hline
\begin{pspicture}
(0,0)(1,1)
\psdots[dotstyle=o, linewidth=1.2pt,linecolor = black, fillcolor = black]%
(0.5,0.5)	
\end{pspicture}		
&
$e^0$
&
Vertex
&
nonempty
\\
    \hline
    \hline
\begin{pspicture}
(0,0)(1,1)
\psline[showpoints=true,linestyle=solid,linecolor = black]%
(0.25,0.25)(0.75,0.75)
\psdots[dotstyle=o, dotsize=1.3pt 2.25,linecolor = black, fillcolor = black]%
(0.25,0.25)(0.75,0.75)
\end{pspicture}&
$e^1$
&
Edge
&
line segment w/o end points
\\
\hline
\hline
\begin{pspicture}
(0,0)(1,1)
\psline*[linecolor = green!50]%
(0.0,0.15)(0.85,0.85)(0.95,0.25)
\psline[linecolor = black]%
(0.0,0.15)(0.85,0.85)(0.95,0.25)(0.0,0.15)
\psdots[dotstyle=o, dotsize=1.3pt 2.25,linecolor = black, fillcolor = black]%
(0.0,0.15)(0.85,0.85)(0.95,0.25)
\end{pspicture}&
$e^2$
&
Filled triangle
&
nonempty triangle interior w/o edges
\\
\hline
\end{tabular}
\end{table}

\begin{remark}
Closure finite cell complexes with weak topology (briefly, CW complexes) were introduced by J.C.H. Whitehead~\cite{Whitehead1939homotopy}, later formalized in~\cite{Whitehead1949BAMS-CWtopology}.  In this work\footnote{Here, we use $cl(e^n)$ (closure of a cell) and $\bdy(e^n)$ (contour of a cell), instead of Whitehead's $\bar{e^n}$ and $\partial(e^n)$.}, a {\bf cell complex} $K$ (or complex)~\cite[\S 4, p. 221]{Whitehead1949BAMS-CWtopology} is a Hausdorff space (union of disjoint open cells $e,e^n,e_i^n$) such that the closure of an $n$ cell $e^n\in K$ (denoted by $\cl(e^n)$ is the image of a map $f:\sigma^n\to \cl(e^n)$, where $\sigma^n$ is a fixed $n$-simplex and where the boundary $\bdy(e^n)$ (otherwise known as the {\bf contour} of a complex) is defined by
\[
\bdy(e^n) = \overbrace{f(\bdy(e^n)) = \cl(e^n) - \Int(e^n).}^{\mbox{\textcolor{blue}{\bf Complex contour $\to$ closure $\cl(e^n)$ minus $Int(e^n)$ interior}}}
\]
Notice that a subcomplex $X\subset K$ has the weak topology, since $X$ is the union of a finite number intersections $X\cap \cl(e)$ for single cells $e\in K$~\cite[\S 5, p. 223]{Whitehead1949BAMS-CWtopology}.  From a geometric perspective, a cell complex is a triangulation of the CW space $K$~\cite[p. 246]{Whitehead1939homotopy}.
\qquad \textcolor{blue}{\Squaresteel}
\end{remark}


\end{appendix}

\bibliographystyle{amsplain}
\bibliography{NSrefs}

\end{document}